\documentclass[11pt]{amsart}
\usepackage{amsmath, amsthm, amssymb}
\usepackage{graphicx}
\usepackage[dvipsnames]{xcolor} 
\usepackage{srcltx} 
\usepackage{verbatim}
\usepackage{bbm}
\usepackage{hyperref}
\hypersetup{
    colorlinks=true,
    linkcolor=blue,
    filecolor=magenta,      
    urlcolor=blue,
    citecolor = blue,
    }
\usepackage{tikz-cd}

\newtheorem{thm}{Theorem}[section]
\newcommand{\bt}{\begin{thm}}
\newcommand{\et}{\end{thm}}

\newtheorem{conj}[thm]{Conjecture}

\newtheorem{ex}[thm]{Example}

\newtheorem{cor}[thm]{Corollary}   
 
\newcommand{\bc}{\begin{cor}}
\newcommand{\ec}{\end{cor}}

\newtheorem{lem}[thm]{Lemma}   
\newcommand{\bl}{\begin{lem}}
\newcommand{\el}{\end{lem}}

\newtheorem{prop}[thm]{Proposition}
\newcommand{\bp}{\begin{prop}}
\newcommand{\ep}{\end{prop}}

\newtheorem{defn}[thm]{Definition}
\newcommand{\bd}{\begin{defn}}    
\newcommand{\ed}{\end{defn}}

\newtheorem{rmrk}[thm]{Remark}   

\newcommand{\br}{\begin{rmrk}}
\newcommand{\er}{\end{rmrk}}

\newtheorem*{*thm}{Theorem}

\newcommand{\be}{\begin{equation}}

\newcommand{\ee}{\end{equation}}

\newcommand{\injrad}{\textrm{injrad}}

\newcommand{\N}{\mathbb{N}}

\newcommand{\R}{\mathbb{R}}

\newcommand{\vare}{\varepsilon}

\newcommand{\diam}{\operatorname{Diam}}

\newcommand{\vol}{\operatorname{Vol}}

\newcommand{\lp}{\left (}
\newcommand{\rp}{\right )}

\newcommand{\Sp}{\mathbb{S}}      
\newcommand{\Tor}{\mathbb{T}}     

\def\Xint#1{\mathchoice
{\XXint\displaystyle\textstyle{#1}}%
{\XXint\textstyle\scriptstyle{#1}}%
{\XXint\scriptstyle\scriptscriptstyle{#1}}%
{\XXint\scriptscriptstyle\scriptscriptstyle{#1}}%
\!\int}
\def\XXint#1#2#3{{\setbox0=\hbox{$#1{#2#3}{\int}$ }
\vcenter{\hbox{$#2#3$ }}\kern-.6\wd0}}

\def\dashint{\Xint-}

\title[Sobolev Inequalities for Riemannian Metrics and Distances]{Sobolev Inequalities and Convergence For Riemannian Metrics and Distance Functions}

\author{B. Allen}
\author{E. Bryden}

\begin{document}

\maketitle

\begin{abstract}
If one thinks of a Riemannian metric, $g_1$, analogously as the gradient of the corresponding distance function, $d_1$, with respect to a background Riemannian metric, $g_0$, then a natural question arises as to whether a corresponding theory of Sobolev inequalities exists between the Riemannian metric and its distance function. In this paper we study the sub-critical case $p < \frac{m}{2}$ where we show a Sobolev inequality exists between a Riemannian metric and its distance function. In particular, we show that an $L^{\frac{p}{2}}$ bound on a Riemannian metric implies an $L^q$ bound on its corresponding distance function. We then use this result to state a convergence theorem and show how this theorem can be useful to prove geometric stability results by proving a version of Gromov's conjecture for tori with almost non-negative scalar curvature in the conformal case. Examples are given to show that the hypotheses of the main theorems are necessary.
\end{abstract}

\section{Introduction}

If one thinks of a Riemannian metric, $g_1$, analogously as the gradient of the corresponding distance function, $d_1$, with respect to a background Riemannian metric, $g_0$, then a natural question arises as to whether a corresponding theory of Sobolev inequalities exists between the Riemannian metric and its distance function. In the previous work of the first named author \cite{Allen} it was observed that if one assumes an $L^{p}$, $p > \frac{m}{2}$ bound on an $m$-dimensional Riemannian metric then one can obtain H\"{o}lder control on  the corresponding distance function. This shows that there is an analogy for Morrey's inequality for Riemannian manifolds. The work of the first named author, R. Perales, and C. Sormani \cite{Allen-Perales-Sormani} shows that in the critical case where one assumes $L^{\frac{m}{2}}$ convergence of the Riemannian manifolds, and other geometric conditions, one can obtain Sormani-Wenger Intrinsic Flat convergence of the sequence. In this paper we are interested in studying the sub-critical case $p < \frac{m}{2}$ where we show a Sobolev theory inequality between a Riemannian metric and its distance function.

To this end, the first main theorem is a Sobolev inequality relating the $L^{\frac{p}{2}}$ norm of a Riemannian metric to the $L^q$ norm of its corresponding distance function.

\begin{thm}\label{Main Sobolev Ineq}
    Let $M^m$ be a connected, closed, and oriented manifold and $(M,g)$ be a fixed Riemannian manifold with metric $g$, and let $h$ be another Riemannian metric on $M$. Then, we have
    \begin{equation}
        \|d_{h}(x,y)\|_{L^{q}(M\times M, g\oplus g)}\le C(M,g,q)\|h\|^{\frac{1}{2}}_{L^{\frac{p}{2}}(M,g)},
    \end{equation}
    with $q<\frac{mp}{m-p}$.
\end{thm}

Then we use an extension of Theorem \ref{Main Sobolev Ineq} to the difference between a sequence of Riemannian manifolds and their limit in order to give the following convergence theorem which guarantees pointwise convergence of distances to the distance function of a limiting Riemannian manifold. We then give an application where the concluded convergence in Theorem \ref{MainTheorem} is the strongest notion of convergence one should expect. We note that one way to think about this theorem is as a tool for bootstrapping up to stronger metric geometry notions of convergence such as Gromov-Hausdorff (GH) or Sormani-Wenger Intrinsic Flat (SWIF) convergence. Hence one can take a suboptimal $L^p$ convergence for a sequence of Riemannain metrics and conclude pointwise convergence of distances off of a singular set and then combine this pointwise convergence of distances with other geometric estimates in order to conclude GH and/or SWIF convergence. In fact, S. Lakzian and C. Sormani \cite{LaS} give conditions which imply GH and/or SWIF convergence when one assumes convergence of Riemannian metrics off of a singular set which one could possibly combine with Theorem \ref{MainTheorem}. We also have examples which show that the notion of convergence concluded in Theorem \ref{MainTheorem} is the best conclusion that one should expect given our hypotheses.

\begin{thm}\label{MainTheorem}
Let $M$ be a connected, closed, and oriented manifold, $g_0$ a smooth Riemannian manifold and $g_j$ a sequence of continuous Riemannian manifolds. If
\begin{align}
    g_j(v,v) \ge \left(1-\frac{1}{j} \right) g_0(v,v), \quad \forall p\in M, v \in T_pM,
\end{align}
and
\begin{align}
    \|g_j-g_0\|_{L^p_{g_0}(M)}\rightarrow 0,\quad \frac{1}{2} \le p\le \frac{m}{2},
\end{align}
then there exists a sub-sequence of the $g_{j}$ and there exists $U \subset M$ measurable so that $\vol_{g_0}(U)=\vol_{g_0}(M)$  and $\mathrm{dim}_{\mathcal{H}_{g_0}}(M \setminus U)\le m-2p$ and
\begin{align}
    d_k(p,q) \rightarrow d_0(p,q), \quad \forall p,q \in U.
\end{align}
Furthermore, for every $\varepsilon > 0$ we can find a measurable set $U_{\varepsilon}\subset M$ so that $\vol_{g_0}(M \setminus U_{\varepsilon})< \varepsilon$ and 
\begin{align}
    d_k(p,q) \rightarrow d_0(p,q),
\end{align}
uniformly $\forall p,q \in U_{\varepsilon}$.
\end{thm}
\begin{rmrk}
We note that although $M \setminus U$ has small $g_0$ volume we are not claiming that it has small $g_j$ volume or even finite $g_j$ volume as $j \rightarrow \infty$. In fact, the geometry of $g_j$ on $M \setminus U$ can be quite exotic and not at all close to the geometry of $g_0$. For instance, splines, bubbles, and infinite cylinders can be forming along the sequence $g_j$ on the set $M \setminus U$. The strength of the theorem is that it can identify a full measure set with respect to $g_0$ where the metric spaces are converging in the pointwise sense. We will see that this is particularly useful for geometric stability results where bubbling can occur along a sequence.
\end{rmrk}

\begin{rmrk}
In the case that $p=\frac{m}{2}$ we see that the singular set must be Hausdorff dimension $0$ which can be seen as a refinement of the main theorem of the first named author, R. Perales, and C. Sormani where volume preserving intrinsic flat convergence is concluded under the same hypotheses \cite{Allen-Perales-Sormani}. In that paper pointwise a.e. convergence in the $p=\frac{m}{2}$ case (in fact the authors only need to assume volume convergence which is implied by $L^{\frac{m}{2}}$ convergence) was shown but there was no estimate given for the dimension of the singular set.
\end{rmrk}

Example \ref{Ex: Dim of Singular Set is Sharp } shows that the dimension of the singular set given in Theorem \ref{MainTheorem} is sharp and we also see that one should not expect to conclude any stronger notion of convergence under these hypotheses. In Example \ref{Ex: Subsequence is Needed} we show why a subsequence is necessary in Theorem \ref{MainTheorem} which is an analagous example to the classical analysis example which shows that if one assumes $L^p$ convergence of a sequence of functions then one can conclude pointwise convergence almost everywhere but necessarily on a subsequence. We also review Example \ref{Ex Cinched-Sphere} from the work of the first named author and C. Sormani \cite{Allen-Sormani-2} which shows that the assumed lower bound on distances is necessary to conclude convergence of distances to the distance function of a Riemannian manifold. Lastly we review Example \ref{Ex L^p Conv} from \cite{Allen-Sormani-2} which shows that if one assumes $L^p$ convergence for $p > \frac{m}{2}$ then one will be able to conclude uniform convergence of distances on the entire manifold. Hence $\frac{m}{2}$ is the critical power which separates when one should expect uniform convergence versus pointwise convergence of distances off of a singular set. The reader may also wish to study other examples comparing and contrasting notions of convergence for Riemannian manifolds given in \cite{Allen-Sormani, Allen-Bryden, Allen-Sormani-2, Allen-Perales-Sormani}.

There has been much research done on using integral bounds on Ricci curvature to obtain $C^{k,\alpha}$ and $W^{k,p}$ bounds on Riemannian manifolds such as Anderson \cite{Anderson-Ricci, Anderson-Orbifold}, Anderson and Cheeger \cite{Anderson-Cheeger}, Cheeger and Colding \cite{Cheeger-Colding-1}, Colding \cite{Colding-shape, Colding-volume}, Gao \cite{Gao-integral1}, Petersen and Wei \cite{Petersen-Wei-integral1, Petersen-Wei-integral2}, and Yang \cite{DYang-integral1, DYang-integral2, DYang-integral3}(See the survey by Petersen \cite{Petersen-Survey} for a broad overview). A result similar to Theorem \ref{MainTheorem}, due to J. Cheeger and T. Colding, can be found in \cite[Thm. 2.11]{Cheeger-Colding-2}. However, the result there has the same exponent on both sides of the inequality, and so is not in the spirit of a Sobolev inequality.  More recently C. Aldana, G. Carron, and S. Tapie \cite{Aldana-Carron-Tapie} used integral bounds on scalar curvature for a sequence of conformal Riemannian manifolds to show Gromov-Hausdorff convergence. In \cite{SH} S. Honda defines a notion of $L^p$ convergence for tensors under Gromov-Hausdorff convergence and shows that for sequences of Riemannian manifolds converging under Gromov-Hausdorff convergence one can always conclude $L^p$ weak convergence of the Remannian metrics. The authors \cite{Allen-Bryden} have also previously studied a similar problem where one assumes $W^{m-1,p}$ control on a Riemannian metric and concludes H\"{o}lder control on the corresponding distance function which is particularly interesting in the case where $1 \le p \le \frac{m}{m-1}$.

As an application of Theorem \ref{MainTheorem} we would like to study the following conjecture by Gromov \cite{GroD}.

\begin{conj}\label{MainConjecture}
Let $M_j = (\mathbb{T}^n, g_j)$, $n \ge 3$ be a sequence of Riemannian manifolds such that 
\begin{align} \label{HypothesisConjecture}
R_{g_j} \ge -\frac{1}{j}, \,\,\, V_0 \le \vol(M_j) \le V^0 \,\,\, \text{and} \,\,\, \diam(M_j) \le D_0  ,
\end{align}
where $R_{g_j}$ is the scalar curvature. 
If no bubbling occurs along the sequence or one can cut out the bubbles forming along the sequence then there is a subsequence of $M_j$ converging in some sense to the flat torus.
\end{conj}

This conjecture by Gromov is intentionally vague and one needs to be more precise about how they will handle bubbling and in what sense the subsequence will converge. In previous special cases studied for this conjecture by the first named author, L. Hernandez-Vazquez, D. Parise, A. Payne, and S. Wang \cite{AHMPPW}, the first named author \cite{Allen, Allen-Tori}, A. J. Cabrera Pacheco, C. Ketterer, and R. Perales \cite{PKP19}, and J. Chu and M.-C. Lee \cite{JCMCL} conditions were imposed, such as the MinA condition, a uniform integrability condition or an $L^p$ bound for large $p$, to rule out the possibility of bubbling and Sormani-Wenger intrinsic flat convergence was concluded. In our case we do not rule out the possibility of bubbling but rather cut out the bubbles by applying Theorem \ref{MainTheorem}. Hence the bubbles will occur on the exceptional set identified in Theorem \ref{MainTheorem} and the geometric stability says that as long as we stay away from the singular set we will find uniform convergence of the metric spaces.

\begin{thm}\label{TorusThm}
Consider a sequence of Riemannian $n-$manifolds $M_j=(\mathbb{T}^n,g_j)$, $n \ge 3$ satisfying 
 \begin{align} \label{HypothesisMainThm}
R_{g_j} \ge -\frac{1}{j}, \,\,\,  \,\,\, \diam(M_j) \le D_0, \,\,\,  \,\,\, \vol(M_j) \le V_0.
\end{align}
Let $g_0$ be a flat torus where $\mathbb{T}^m_0=(\mathbb{T}^m,g_0)$ and assume that $M_j$ is conformal to $\tilde{M}_{0,j}=(\mathbb{T}^m,\tilde{g}_{0,j})$, a metric with constant zero or negative scalar curvature and unit volume, i.e.  $g_j = e^{2f_j} \tilde{g}_{0,j}$. Furthermore, assume that
\begin{align}
\tilde{g}_{0,j} \rightarrow g_0 \text{ in } C^1,
\end{align}
and
\begin{align}
\int_{\mathbb{T}^m} e^{-2f_j} d V_{\tilde{g}_{0,j}} \le C,\label{LowerVolume}
\end{align} 
then for any $\varepsilon >0$ there exists a subsequence and a set $U_{\varepsilon} \subset \Tor^m$ so that  $\vol(U_{\varepsilon})=\vol(\Tor^m)-\varepsilon$ and
\begin{align}
  d_k(q_1,q_2) \rightarrow d_{\infty}(q_1,q_2),
\end{align}
uniformly on $U_{\varepsilon}$ where $d_{\infty}$ is the distance function for $\bar{\mathbb{T}}_0^m = (\Tor^m,\bar{g}_0=c_{\infty}^2g_0)$, $c_{\infty}^2=\displaystyle\lim_{k \rightarrow \infty}(\overline{e^{-f_k}})^{-2}= \lim_{k \rightarrow \infty} \left( \dashint_{\mathbb{T}^m} e^{-f_k}dV_{\tilde{g}_{0,j}} \right)^{-2}$.
\end{thm}
\begin{rmrk}
We note that although $\vol(\Tor^m \setminus U_{\varepsilon})$ has small volume with respect to the flat torus one should not expect it to have small or even finite volume with respect to the sequence $M_j$. This is because geometrically bubbles are forming along the sequence on the singular set which will add more volume to $M_j$ and possibly infinitely more volume as $j \rightarrow \infty$. Theorem \ref{TorusThm} motivates an alternative approach to Conjecture \ref{MainConjecture} where one tries to cut out the bubbles occurring along some exceptional set and show uniform convergence or Gromov-Hausdorff convergence on the complement of a neighborhood of the exceptional set. 
\end{rmrk}

\begin{rmrk}
Notice that \eqref{LowerVolume} is equivalent to $\|g_j^{-1}\|_{L^1(M,\tilde{g}_{0,j})}\le C$ which is a non-collapsing assumption for the sequence. 
\end{rmrk}

\section{Background}

In this section we review the tools from the first named author's paper \cite{Allen} which we will use throughout the rest of the paper. We start by stating a simple lemma which allows us to estimate the norm of vectors with respect to one Riemannian metric to the norm of vectors with respect to another Riemannian metric.

\begin{lem}\label{NormComparisonLemma}
 For $v \in T_pM$ and $g_0,g_1$ Riemannian metrics on $M$ we have the inequality
 \begin{align}
 |v|_{g_1}\le |g_1|_{g_0}^{\frac{1}{2}}|v|_{g_0}.
 \end{align}
 \end{lem}
 
 This can be useful to estimate lengths of curves with respect to $g_1$ by the lengths of curves with respect to $g_0$ as long as one adds the norm of $g_1$ with respect to $g_0$. We now also review the definition of the $L^p$ norm of $g_1$ with respect to $g_0$:
 
 \begin{align}
     \|g_1\|_{L^p_{g_0}(M)}=\left(\int_M |g_1|_{g_0}^pdV_{g_0}\right)^{1/p},
 \end{align}
 where $dV_{g_0}$ is the volume form with respect to $g_0$.

 If $f_1:M\rightarrow [0,\infty)$ and $g_1=f_1^2g_0$ then factoring $f_1^2-1$ and applying H{\"o}lder's inequality, we estimate $\|g_1-g_0\|_{L_{g_0}^{\frac{p}{2}}(M)}$ in terms of $\|f_1-1\|_{L_{g_0}^{p}(M)}$
  \begin{align}
     \|g_1-g_0\|_{L^{\frac{p}{2}}_{g_0}(M)}&=\left(\int_M |g_1-g_0|_{g_0}^{\frac{p}{2}}dV_{g_0}\right)^{{\frac{2}{p}}}
     \\&=\left(\int_M |f_1^2-1|_{g_0}^{\frac{p}{2}}|g_0|_{g_0}^{\frac{p}{2}}dV_{g_0}\right)^{{\frac{2}{p}}}
     \\&=n\left(\int_M |f_1-1|_{g_0}^{\frac{p}{2}}|f_1+1|_{g_0}^{\frac{p}{2}}dV_{g_0}\right)^{{\frac{2}{p}}}\label{eq:conf_metr_factor}
     \\&\le n\left(\int_M |f_1-1|_{g_0}^{p}dV_{g_0}\right)^{{\frac{1}{p}}}\left( \int_M|f_1+1|_{g_0}^{p}dV_{g_0}\right)^{{\frac{1}{p}}},\label{LastEqL^pExplanation}
 \end{align}
 where the last term in \eqref{LastEqL^pExplanation} in bounded by the Minkowski inequality and an $L^p$ bound on $f_1$. On the otherhand, the reverse triangle inequality gives us that
 \begin{equation}
     \|g_1-g_0\|_{L^{\frac{p}{2}}_{g_0}}\geq\|g_1\|_{L^{\frac{p}{2}}_{g_0}}-n|M|^{\frac{2}{p}}=n\|f_1\|_{L^{p}_{g_0}}-n|M|^{\frac{2}{p}}.
 \end{equation}

 We now give the definition of a symmetric family of curves, depicted in Figure \ref{fig:SymmetricFamilyCurves}, which foliates a region of full volume by curves connecting points $q_1,q_2 \in M$.

\begin{defn}\label{SymFamilyofCurvesDef}
 Let $q_1,q_2 \in M_0$ and $\alpha(t)$ be a length minimizing geodesic joining $q_1$ to $q_2$, with respect to $M_0$, of length $L$. By extending a distance $\tau \in (0,\varepsilon)$ in radial directions orthogonal to $\alpha'(t)$, which can be parameterized over $\Sp^{m-2}$, one obtains a tubular neighborhood $\alpha \subset \bar{U}_{\varepsilon} \subset M$ with coordinates $(\tau,t,\vec{s}) \in [0,\varepsilon]\times [0,L]\times\Sp^{m-2}$. For fixed $(\tau,\vec{s}) \in  [0,\varepsilon] \times \Sp^{m-2}$ we define a curve connecting $q_1$ and $q_2$ in coordinates by
 \begin{align}
 \gamma(\tau,t,\vec{s}) = (\tau L\sin\left(\frac{\pi}{L}t\right),t,\vec{s}).
 \end{align} 
 We define the \textbf{symmetric family of curves joining $q_1$ to $q_2$ of width $\varepsilon$} to be the set $U_{\varepsilon}$ foliated by the curves $\gamma(\tau,t,\vec{s})$ for $(\tau,t,\vec{s}) \in (0,\varepsilon)\times [0,L]\times\Sp^{m-2}$.

 Let $d\gamma^{-1}$ be the differential of the map $\gamma^{-1}:U_{\vare}\setminus\{q_1,q_2\}\rightarrow[0,\vare]\times(0,L)\times\Sp^{m-2}$ and  $[\gamma'(p)]^{\perp}\subset T_pM$ be the orthogonal subspace to $\gamma'(p)$. The \textbf{normal Jacobian} of $\gamma^{-1}$ at $p$ is the determinant of the restriction of $d\gamma^{-1}\big\vert_{p}$ to $[\gamma'(p)]^{\perp}$ denoted
 \begin{align}
    NJ\gamma^{-1}(p):= \det_{[\gamma'(p)]_{g_0}^{\perp}}\left(d\gamma^{-1}\big\vert_{p}\right).
 \end{align}
 \end{defn}
 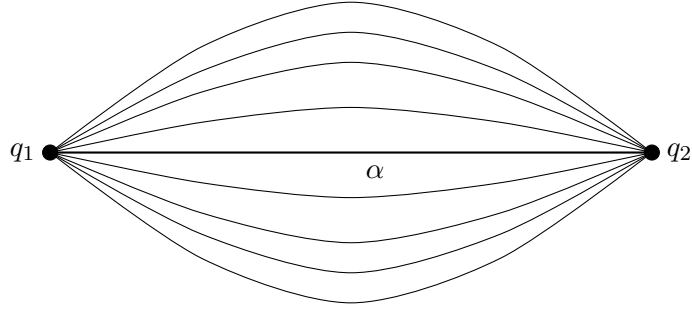
\begin{figure}
 \label{fig:SymmetricFamilyCurves}
 \begin{tikzpicture}[scale=4]
  \draw[thick] (-1,1) -- (1,1);
  \draw[fill] (-1,1) circle [radius=0.025];
  \node[left, outer sep=2pt] at (-1,1) {$q_1$};
  \draw[fill] (1,1) circle [radius=0.025];
  \node[right, outer sep=2pt] at (1,1) {$q_2$};
  \node[right, outer sep=2pt] at (0,1-1/16) {$\alpha$};
   \draw plot [smooth] coordinates {(-1,1)(-.5,1.1)(0,1.15)(.5,1.1)(1,1)};
  \draw plot [smooth] coordinates {(-1,1)(-.5,1-.1)(0,1-.15)(.5,1-.1)(1,1)};
  \draw plot [smooth] coordinates {(-1,1)(-.5,1.2)(0,1.3)(.5,1.2)(1,1)};
  \draw plot [smooth] coordinates {(-1,1)(-.5,1-.2)(0,1-.3)(.5,1-.2)(1,1)};
  \draw plot [smooth] coordinates {(-1,1)(-.5,1.27)(0,1.4)(.5,1.27)(1,1)};
  \draw plot [smooth] coordinates {(-1,1)(-.5,1-.27)(0,1-.4)(.5,1-.27)(1,1)};
  \draw plot [smooth] coordinates {(-1,1)(-.5,1.35)(0,1.5)(.5,1.35)(1,1)};
  \draw plot [smooth] coordinates {(-1,1)(-.5,1-.35)(0,1-.5)(.5,1-.35)(1,1)};
 \end{tikzpicture} 
  \caption{Symmetric family of curves joining $q_1$ to $q_2$ which foliates a region of full volume around $\alpha$. }
 \end{figure}
 
 We then review a Lemma which guarantees the existence of a normal cylindrical neighborhood around any geodesic with a uniform size which does not depend on the points $q_1,q_2 \in M$ which the geodesic is connecting. 
 
 \begin{lem}\label{NorNeighRadiusEst}
 Let $M_0=(M,g_0)$ be a smooth Riemannian manifold defined on the closed manifold $M$. Let $K\in [0,\infty)$ be a bound on the absolute value of the sectional curvature of $M_0$ and $\injrad(M_0)$ be the injectivity radius. If  $q_1,q_2\in M_0$ and $\alpha(t)$ a length minimizing geodesic joining $q_1$ to $q_2$, parameterized by arc length, then there exists a $\bar{\varepsilon}(K,\injrad(M_0)) > 0$ so that for any $0 < \varepsilon < \bar{\varepsilon}$  a tubular neighborhood of width $\varepsilon$ exists, $\bar{U}_{\varepsilon}$.
 \end{lem}
 
 Lastly, we state a lemma which gives estimates on the geometry of the symmetric family of curves in terms of geometric bounds on the background Riemannian metric $g_0$.
 
 \begin{lem}\label{MainToolsForProof}
 Let $M_0=(M,g_0)$ be a smooth Riemannian manifold defined on the smooth, connected, closed manifold $M$. Assume  $q_1,q_2\in M_0$ and $\alpha(t)$ a length minimizing geodesic joining $q_1$ to $q_2$, parameterized by arc length. then there exists an $\varepsilon >0$ so that a symmetric family of curves joining $q_1$ to $q_2$ of width $\varepsilon$ exists, $U_{\varepsilon}$. If $K \in [0,\infty)$, is a bound on the absolute value of the sectional curvature, $I_0=\injrad(M_0)$ the injectivity radius, and $D_0=\diam(M_0)$ the diameter of $M_0$ then for $\bar{\varepsilon}=\bar{\varepsilon}(K, \injrad(M_0))$ from Lemma \ref{NorNeighRadiusEst} we have that for $0<\varepsilon \le \frac{\bar{\varepsilon}}{\diam(M_0)} $ a symmetric family of curves joining $q_1$ to $q_2$ of width $\varepsilon$ exists, $U_{\varepsilon}$,  and with $\gamma'=\frac{d\gamma}{dt}$ we find
 \begin{align}
 |\gamma'(\tau,t,\vec{s})|_{g_0} \le C(K,I_0,D_0), \quad (\tau,t,\vec{s}) \in [0,\varepsilon]\times [0,L]\times \Sp^{m-2}.
\end{align}  
We find the normal jacobian of $\gamma^{-1}$ to be estimated by
\begin{align}\label{Curve_Family_Coarea_Formula}
NJ\gamma^{-1}&\le C(M,g_0) \left(L\sin\left(\frac{\pi}{L}t\right) \right)^{1-m}.
\end{align}

 \end{lem}
 
 Finally, we state the main theorem of the first named author in \cite{Allen} which is analogous to Theorem \ref{MainTheorem}. These two theorems together show the analogy of Sobolev theory applied to Riemannian metrics and their distance functions does hold.
 
 \begin{thm}\label{ConfDistBoundAbove}
Let $g_0$ be a smooth Riemannian metric and $g_1$ a continuous Riemannian metric defined on the smooth, connected, closed  manifold $M$. If
\begin{align}
\|g_1\|_{L_{g_0}^{\frac{p}{2}}(M)} \le C,\quad p > m,
\end{align}
then we can bound the distance between $q_1,q_2$, with respect to $g_1$, by
\begin{align}
d_{g_1}(q_1,q_2) \le C'(M,g_0,C) d_{g_0}(q_1,q_2)^{\frac{p-m}{p}}.
\end{align}
\end{thm}

\section{Examples}

In this section we give several examples which show that the hypotheses of our main theorems are necessary. These examples also further serve to explain the geometric behavior expected under such hypotheses. Our first example shows that the dimension of the singular set in Theorem \ref{MainTheorem} is sharp for $p> 1$.

\begin{ex}\label{Ex: Dim of Singular Set is Sharp } 
Let $(\theta_1,...,\theta_m)$ be coordinates on $\Tor^m$, $n \in \N$, and define
\begin{align}
    \Sigma^n=\{(\theta_1,...,\theta_m) \in \Tor^m: \theta_{n+1}=\theta_{n+2}=...=\theta_m=0\}.
\end{align}
Let $r$ denote the distance to $\Sigma^{m-p}$ with respect to the metric $d\theta_1^2+\dots+ d\theta_m^2$ and consider the sequence of functions on $\Tor^m$, $p \in \N$, $p \not = 1$ given as follows:
\begin{equation}
f_j(r)=
\begin{cases}
\frac{j^{\eta}}{1+\eta\ln(j)} &\text{ if } r \in [0,1/j^{\eta}]
\\ \frac{1}{r(1-\ln(r))} &\text{ if } r \in (1/j^{\eta},1/j]
\\h_j(jr) &\text{ if } r \in [1/j,2/j]
\\ 1 & \text{ otherwise } 
\end{cases}
\end{equation}
where $\eta > 1$ and $h_j:[1,2] \rightarrow \R$ is a smooth, decreasing function so that $h_j(1) = \frac{j}{1+\ln(j)}$ and $h_j(2) = 1$. Then $M_j = (\Tor^m, f_j^2 g_{\Tor^m})$ converges to $\Tor^m$ in $L^q(\Tor^m)$, $\forall 1< q \le p$ and the distance function converge pointwise
\begin{align}
    d_j(q_1,q_2) \rightarrow d_{\Tor^m}(q_1,q_2), \quad \forall q_1,q_2 \in \Tor^m \setminus \Sigma^{m-p},
\end{align}
but the distance function $d_j(q_1, \cdot)$, $q_1 \in \Sigma^{m-p}$ does not converge pointwise.
\end{ex}
\begin{proof}
First we check the claim about $L^q$, $1<q \le p$ convergence where we just need to calculate the $L^p$ norm by H\"{o}lder's inequality on a compact set
    \begin{align}
&C(p,q,\Tor^m)\|f_j-1\|_{L^q(\Tor^m)}\le\|f_j-1\|_{L^p(\Tor^m)} 
\\&\le Vol(T_{\frac{1}{j^{\eta}}}(\Sigma^{m-p})) \left(\frac{j^{\eta}}{1+\eta\ln(j)}-1\right)^p +C_m\int^{1/j}_{1/j^\eta}\frac{r^{p-1}}{r^p(1-\ln(r))^p}dr
\\&= C_m \frac{\left(\frac{j^{\eta}}{1+\eta\ln(j)}-1\right)^p }{j^{\eta p}}+C_m\int_{1-\ln(1/j^{\eta})}^{1-\ln(1/j)}-\frac{1}{u^p} du 
\\&\le C_m \frac{j^{ \eta(p-p)}}{(1+\eta\ln(j))^p} +C_m ((1-\ln(1/j^{\eta}))^{1-p}- (1-\ln(1/j))^{1-p} )
\\&\rightarrow 0, \quad 1< p.
\end{align}
\begin{align}
\|f_j\|_{L^q(\Tor^m)} &\ge Vol(T_{\frac{1}{j}}(\Sigma^{m-p})) \left(\frac{j}{1+\eta\ln(j)}\right)^q 
\\&= C_m \frac{\left(\frac{j}{1+\eta\ln(j)}\right)^q }{j^p} \rightarrow \infty, \quad q > p.
\end{align}
Now notice that since $f_j \ge 1$ that we know that 
\begin{align}
    d_j(q_1,q_2) \ge d_{\Tor^m}(q_1,q_2), \quad \forall q_1,q_2 \in \Tor^m.
\end{align}
For $q_1,q_2 \in \Tor^m \setminus \Sigma^{m-p}$, $p \not =1$ we note that the flat torus geodesic between $q_1, q_2$ is contained inside $ \Tor^m \setminus \Sigma^{m-p}$ for most pairs of points.
So for $j$ chosen large enough this geodesic is contained in $\Tor^m \setminus T_{\frac{1}{j}}(\Sigma^{m-p})$ and hence
\begin{align}
    d_j(q_1,q_2) \le d_{\Tor^m}(q_1,q_2), \forall q_1,q_2 \in \Tor^m \setminus \Sigma^{m-p}.
\end{align}
In the case where the flat torus geodesic between $q_1,q_2$ is not contained in $ \Tor^m \setminus \Sigma^{m-p}$ then for any $\varepsilon > 0$ one can find a piecewise flat torus geodesic curve $\gamma_{\varepsilon}$ so that 
\begin{align}
    d_j(q_1,q_2)\le L_{g_j}(\gamma_{\varepsilon})=d_{\Tor^m}(q_1,q_2)+\varepsilon
\end{align}
 for $j$ chosen large enough and hence we have the desired pointwise convergence of distances.
If we take $q_1 \in \Sigma^{m-p}$ and $q_2$ a radial distance of $1$ from $q_1$ then we see that by definition of $f_j$ the straight line path is the shortest path between $q_1$ and $q_2$ and we can calculate
\begin{align}
d_j(q_1,q_2) &= \int_0^1f_j dr \label{DiamBoundEx3.4}
\\&= \frac{1}{1+\eta\ln(j)}+\int_{1/j^{\eta}}^{1/j} \frac{1}{r(1-\ln(r))} dr 
\\&\qquad+\int_{1/j}^{2/j} h_j(jr) dr + (1 - 2/j) 
\\&=\frac{1}{1+\eta\ln(j)}+ (1 - 2/j) +\int_{1/j}^{2/j} h_j(jr) dr
\\&\qquad+\ln(1-\ln(1/j^{\eta})) - \ln(1-\ln(1/j)) 
\\&\ge \frac{1}{1+\eta\ln(j)}+ (1 - 2/j) 
+\ln\lp\frac{1-\ln(1/j^{\eta})}{1-\ln(1/j)} \rp \rightarrow 1 + \ln(\eta).
\end{align}
A similar calculation for points $q_1 \in \Sigma^{m-p}$ and $q_2 \in \Tor^m$ shows a similar discrepancy in distances and
so we see that we do not find pointwise convergence of distances for $q_1 \in \Sigma^{m-p}$ and $q_2 \in \Tor^m$.
\end{proof}

Now we see an example which shows that a subsequence is necessary in Theorem \ref{MainTheorem}.

\begin{ex}\label{Ex: Subsequence is Needed}
Define the set
\begin{align}
    Q_k= \left\{ \lp\frac{n}{2^{k+1}}, \frac{m}{2^{k+1}}\rp: n,m = 1...2^k\right\},
\end{align}
and define a sequence of points $p_j$ by listing the points in $Q_1$ first, then the points in $Q_2$, then $Q_3$, etc. Let $r_j$ be a sequence where for each $p_j$ we define $r_j$ to be $\frac{1}{2^{k-2}}$ where $p_j$ belonged to the set $Q_k$.
Consider $M_j=(\Tor^2,g_j=f_j^2g_0)$ with conformal factor radially defined from the point $p_j$ to be
\begin{align}
f_j(r)=
    \begin{cases}
    r_j^{-1} & \text{ on } B_{r_j}(p_j)
    \\ 1 &\text{ elsewhere}
    \end{cases}.
\end{align}
Then $f_j \rightarrow 1$ in $L^p(\Tor^2)$,  $1 \le p < 2$ and for all $q_1,q_2 \in M$ we find that $d_j(q_1,q_2)$ does not converge pointwise to $d_{\Tor^2}(q_1,q_2)$ but that on a subsequence we find
\begin{align}
    d_k(q_1,q_2) \rightarrow d_{\Tor^2}(q_1,q_2), \quad \forall q_1,q_2 \in \Tor^2\setminus \{(1/2,1/2)\}.
\end{align}

\end{ex}
\begin{proof}
Notice that for every $p \in \Tor^2$ and any $n \in \N$ we can find a $j > n$ so that $p \in B_{\frac{r_j}{2}}(p_j)$ where $f_j$ is equal to $r_j^{-1}$ on $B_{r_j}(p_j)$. Then we can calculate that
\begin{align}
    d_j(p,p_j) \ge r_j^{-1} d_{\Tor^2}(p,p_j),
\end{align}
since one would connect $p$ to $p_j$ by a radial line where the conformal factor is equal to $r_j^{-1}$. In fact, for any $q \in B_{\frac{r_j}{2}}(p_j)$ we find a similar estimate for $d_j(p,q)$ and for $q \not \in B_{\frac{r_j}{2}}(p_j)$ the geodesic connecting $p$ to $q$ will have a portion which is contained in $B_{r_j}(p_j)$ where this estimate will hold and hence we do not find pointwise convergence of distances.  

    We can define a subsequence by choosing $j_k$ so that $p_{j_k}=(1/2,1/2)$ and then for every $q_1,q_2 \in \Tor^2\setminus \{(1/2,1/2)\}$ the flat torus geodesic is contained in $\Tor^2 \setminus B_{r_k}(p_k)$ for $k$ chosen large enough, or can be approximated by piecewise straight segments contained in $\Tor^2 \setminus B_{r_k}(p_k)$ for $k$ chosen large enough with length arbitrarily close to the flat torus distance between $q_1$ and $q_2$, which implies that 
    \begin{align}
    d_k(q_1,q_2) \rightarrow d_{\Tor^m}(q_1,q_2), \quad \forall q_1,q_2 \in \Tor^2\setminus \{(1/2,1/2)\}.
\end{align}

\end{proof}

We now review an example given by the first named author and C. Sormani in \cite{Allen-Sormani-2} which shows that the lower bound on distances assumed in Theorem \ref{MainTheorem} is necessary to converge to a distance function for a Riemannian metric. This example is quite general since $L^p$ convergence of Riemmanian metrics allows the metrics to disagree on sets of measure zero and curves are measure zero sets. If a shortcut develops along a curve, in a similar way to the following example, then the limit will be a metric space despite the fact that the sequence is converging in $L^p$ to a Riemannian metric. We note that one can also show the following examples converges to the limiting metric space in the volume preserving Sormani-Wenger Intrinsic Flat sense and the measured Gromov-Hausdorff sense but we want to focus on the convergence of distance functions in this paper.

 \begin{ex} \label{Ex Cinched-Sphere}  
Define functions radially from the north pole of a sphere $\Sp^m$ by  
 \be
 f_j(r)=
 \begin{cases}
 1 & r\in[0,\pi/2- 1/j]
 \\  h(jr-\pi/2) & r\in[\pi/2- 1/j, \pi/2+ 1/j]
 \\ 1 &r\in [\pi/2+ 1/j, \pi]
 \end{cases}
\ee
where $h:[-1,1]\rightarrow \R$ is a smooth even function such that 
$h(-1)=1$ with $h'(-1)=0$, 
decreasing to $h(0)=h_0\in (0,1)$ and then
increasing back up to $h(1)=1$, $h'(1)=0$. Then $M_j = (\Sp^m, f_j^2 g_{\Sp^m})$
\begin{align}
d_j(q_1,q_2) \rightarrow d_{\infty}(q_1,q_2), 
\end{align}
uniformly for all $q_1,q_2 \in M$ but $M_{\infty}$ is not isometric to $\Sp^m$. Instead $M_{\infty} = (\Sp^m, f_{\infty}^2 g_{\Sp^m}) $ is the conformal metric with conformal factor
 \be
 f_{\infty}(r)=
 \begin{cases}
 h_0 & r=\pi/2
 \\  1 &\text{ otherwise}
 \end{cases}.
 \ee
 The distances between pairs of points near the equator in this limit metric space is achieved by geodesics which run to the
 equator, and then around inside the cinched equator, and then out again.  
\end{ex}

We also review the following example of the first named author and C. Sormani \cite{Allen-Sormani-2} which shows that when one assumes $L^{p}$ convergence for $p > \frac{m}{2}$ then one should expect uniform convergence of the metric spaces. This example motivated the Morrey type inequality of the first named author \cite{Allen} which shows that one obtains H\"{o}lder control on distances when one assumes a $L^{p}$ bound for $p > \frac{m}{2}$ on a Riemannian metric. We note that one can also show the following example converges to the flat torus in the volume preserving Sormani-Wenger Intrinsic Flat sense and the measured Gromov-Hausdorff sense but we want to focus on the convergence of distance functions in this paper.

\begin{ex}\label{Ex L^p Conv}
Define a sequence of functions on $\Tor^m$ radially from a point $p \in \Tor^m$ by
\begin{equation}
f_j(r)=
\begin{cases}
j^{\alpha} &\text{ if } r \in [0,1/j]
\\h_j(jr) &\text{ if } r \in [1/j,2/j]
\\ 1 & \text{ otherwise } 
\end{cases}
\end{equation}
where $0< \alpha < 1$ and $h_j:[1,2] \rightarrow \R$ is a smooth, decreasing function so that $h_j(1) = j^{\alpha}$ and $h_j(2) = 1$. Then for $M_j = (\Tor^m, f_j^2 g_{\Tor^m})$ 
\begin{align}
    \|f_j-1\|_{L^p(\Tor^m)} \rightarrow 0 \text{ for }p < \frac{m}{\alpha},
\end{align}
 \begin{align}
    1 \le f_j,
 \end{align}
and
\begin{align}
d_j(q_1,q_2)\rightarrow d_{\Tor^m}(q_1,q_2)
\end{align}
uniformly for all $q_1,q_2 \in M$.
\end{ex}

\section{Poincare Inequalities Relating Riemannian Metrics to Distance}

In order to obtain a Poincare like estimate for metrics, we follow by analogy the proof of the Poincare inequality given on page 29 of \cite{H}. In order to do this, we will estimate the normal Jacobian term in Lemma \ref{MainToolsForProof} by a distance term:
\begin{lem}\label{Coarea Formula to Potential Integral}Let $U_{\vare}$ be a symmetric family of curves connecting two points of $M$, say $x$ and $y$. Let $\Gamma_{\vare}=\left([0,\vare]\times\mathbb{S}^{m-2},d\tau^2+\tau^2g_{\Sp^{m-2}}\right)$, which parameterizes the curves foliating $U_{\vare}$. Let $\mu_{\vare}$ denote the volume measure for $\Gamma_{\vare}$. Then, for any non-negative measurable function, say $f$, we have
\begin{equation}
    \int_{\Gamma_{\vare}\times[0,L]}fdtd\mu_{\vare}\le C(M,g_{0})\int_{U_{\vare}} f(z)\left(d_{g_0}(x,z)^{1-m}+d_{g_{0}}(y,z)^{1-m}\right)dV_{g_{0}}(z)
\end{equation}

\end{lem}
\begin{proof}
    Consider map from $U_{\vare}$ to $\Gamma_{\vare}$ which takes a point in $z\in U_{\vare}$ to the curve $\gamma(t)\in\Gamma_{\vare}$ it belongs to. This defines a function $t:U_{\vare}\rightarrow [0,L]$, defined by the equation $\gamma(t(z))=z$. Equation (\ref{Curve_Family_Coarea_Formula}) in Lemma \ref{MainToolsForProof} establishes the necessary coarea formula for these maps to relate the volume measure of $U_{\vare}$ to the measure $dt\otimes\mu_{\vare}$ on $[0,L]\times\Gamma_{\vare}$: for any Borel measurable function, say $f$, we have
    \begin{equation}
            \int_{\Gamma_{\vare}\times[0,L]}fdtd\mu_{\vare}(\eta)\le C(M,g_0)L^{1-m}\int_{U_{\vare}}f\sin\left(\frac{\pi}{L}t(z)\right)^{1-m}dV_{g_{0}}(z).
    \end{equation}
    From Definition \ref{SymFamilyofCurvesDef} we observe that for $z\in U_{\vare}$, we may express it in coordinates $(t,\tau,\theta)$. Recall that the coordinate $t$ is precisely the length that we have traveled along the geodesic connecting $x$ to $y$, and $\tau\sin\left(\frac{\pi}{L}t\right)$ is precisely the distance $z$ is away from the geodesic. Thus, from the triangle inequality we obtain
    \begin{equation}
        d_{g_{0}}(x,z)\le t(z)+\tau\sin\left(\frac{\pi}{L}t(z)\right),
    \end{equation}
    and
    \begin{equation}
        d_{g_{0}}(y,z)\le(L-t(z))+\tau\sin\left(\frac{\pi}{L}(L-t(z))\right).
    \end{equation}
    Let us focus on the first inequality. Since we want a comparison with the normal Jacobian in Lemma \ref{MainToolsForProof}, we will factor out the $\sin{\frac{\pi}{L}}$ term to obtain
    \begin{equation}
        d_{g_{0}}(x,z)\le\sin\left(\frac{\pi}{L}t(z)\right)\left(\frac{t(z)}{\sin\left(\frac{\pi}{L}t(z)\right)}+\tau\right).
    \end{equation}
    Now, let us assume for the moment that the $t$-coordinate of $z$ is in the range $[0,\frac{L}{2}]$. Then, we have the following bound
    \begin{equation}
        \frac{t(z)}{\sin\left(\frac{\pi}{L}t(z)\right)}\le\frac{L}{2}.
    \end{equation}
    In particular, we obtain
    \begin{equation}
        d_{g_{0}}(x,z)\le C L\sin\left(\frac{\pi}{L}t(z)\right).
    \end{equation}
    Similar reasoning shows that for $z$ with $t$-coordinate in the range $[\frac{L}{2},L]$ we have
    \begin{equation}
        d_{g_{0}}(y,z)\le C L\sin\left(\frac{\pi}{L}(L-t(z))\right).
    \end{equation}
    In particular, we see that
    \begin{equation}
        L^{1-m}\sin\left(\frac{\pi}{L}t(z)\right)^{1-m}\leq C d_{g_{0}}(x,z)^{1-m}
    \end{equation}
    for $z$ with $t$-coordinate in the range of $[0,\frac{L}{2}]$, and
    \begin{equation}
        L^{1-m}\sin\left(\frac{\pi}{L}t(z)\right)^{1-m}\leq C d_{g_{0}}(y,z)^{1-m}
    \end{equation}
    for $z$ with $t$-coordinate in the range of $[\frac{L}{2},L]$.
    If we add the right hand side of the above two inequalities together, then we see that for $z$ with any $t$-coordinate we have
    \begin{equation}
        L^{1-m}\sin\left(\frac{\pi}{L}t(z)\right)^{1-m}\leq C\left(d_{g_{0}}(x,z)^{1-m}+d_{g_{0}}(y,z)^{1-m}\right).
    \end{equation}
    Plugging the above inequalities into the coarea formula given by the normal Jacobian in Lemma \ref{MainToolsForProof} gives the result.
\end{proof}
\begin{lem}\label{lem: distance potential estimate}
    Let $(M,g_{0})$ be some fixed $m$-dimensional Riemannian manifold with metric $g_{0}$, and let $h$ be another Riemannian metric on $M$. Finally, let $d_{h}$ be the distance function determined by $h$. Then, we have
    \begin{equation}
        d_h(x,y)\le C(M,g_0)\left(\int_{M}\frac{|h|^{\frac{1}{2}}_{g_{0}}dV_{g_0}(z)}{d_{g_{0}}(x,z)^{m-1}}+\int_{M}\frac{|h|^{\frac{1}{2}}_{g_{0}}dV_{g_0}(z)}{d_{g_{0}}(y,z)^{m-1}}\right).
    \end{equation}
\end{lem}
\begin{proof}
    Let $U_{\vare}$ be a symmetric family of curves with width $\vare$ connecting $x$ to $y$, and let $\Gamma$ be the $m-1$ Euclidean ball of radius $\vare$ which parameterizes the curves foliating $U_{\vare}$. 
    Since every curve in $\Gamma_{\vare}$ connects $x$ to $y$, we have for all $\eta$ in $\Gamma_{\vare}$
    \begin{equation}
        d_{h}(x,y)\leq\int_{[0,L]}\sqrt{h(\Dot{\eta},\Dot{\eta})}dt\le\int_{[0,L]}|h|_{g_{0}}^{\frac{1}{2}}|\Dot{\eta}|_{g_{0}}dt.
    \end{equation}
    Since $\mu_{\vare}$ is the volume measure on $\Gamma_{\vare}$, integrating both sides of the above by $\mu_{\vare}$ gives
    \begin{equation}
        \omega_{m-1}\vare^{m-1}d_{h}(x,y)\leq\int_{\Gamma_{\vare}\times[0,L]}\sqrt{h(\Dot{\eta},\Dot{\eta})}dtd\mu(\eta)\leq\int_{\Gamma_{\vare}}\int_{[0,L]}|h|_{g_{0}}^{\frac{1}{2}}|\Dot{\eta}|_{g_{0}}dtd\mu(\eta).\label{EqImportant}
    \end{equation}
    Plugging \eqref{EqImportant} into Lemma \ref{Coarea Formula to Potential Integral}, fixing an $\vare>0$, and using the fact that $|\Dot{\eta}|_{g_0}$ is uniformly bounded above, we get
    \begin{align}
        d_{h}(x,y)
        \leq  C(M,g_{0})\left(\int_{M}\frac{|h|^{\frac{1}{2}}_{g_{0}}dV_{g_0}(z)}{d_{g_{0}}(x,z)^{m-1}}+\int_{M}\frac{|h|^{\frac{1}{2}}_{g_{0}}dV_{g_0}(z)}{d_{g_{0}}(y,z)^{m-1}}\right).
    \end{align}
\end{proof}

In fact, if we are content to add an error term, which will depend on the width of the symmetric family of curves, and to assume $h\geq g_{0}$, then we can estimate $d_{h}(x,y)-d_{g_{0}}(x,y)$ within this error term. This process will involve demonstrating that there is a uniform relationship between the width of a symmetric family of curves, and the degree to which each member of the family approximates the geodesic at the center of the family. In order to establish this, we need a few properties of tubular neighborhoods about geodesics. The following is a direct application of Corollary 1.31 in Chapter 1 of \cite{Cheeger-Ebin}
\begin{lem}\label{uniform distance estimate symmetric family of curves}
    Let  $(M,g_0)$ be a  compact Riemannian manifold. Then, for any $\vare>0$ there exists a $\delta(M,g_0,\vare):=\delta=>0$ such that if $U_{\delta}$ is a symmetric family of curves about a minimal length path connecting any two points, say $x$ to $y$, then for any curve in the family $U_{\delta}$, say $c$, we have
    \begin{equation}
        l(c)\leq d_{g_{0}}(x,y)+2\vare.
    \end{equation}
\end{lem}

\begin{proof}
Since $(M,g_0)$ is compact there exist $k \in \R$ so that we have the following sectional curvature lower bound
    \begin{equation}
        k\leq\mathrm{Sec}(g_{0}).
    \end{equation}
    Let $M_k$ be the model space of constant curvature $k$ and denote the symmetric families of curves in $M_k$ by $\tilde{U}_{\delta}$.
    By Corollary 1.31 in Chapter 1 of \cite{Cheeger-Ebin}, the length of one of the curves foliating $U_{\delta}$, say $c_{\tau}$, is bounded by the length of the corresponding curve of the form
    \be
        \tilde{c}_{\tau}(t):=\exp\left(\tilde{\gamma}(t),\tau\sin(t)\theta\right)
    \ee
    in $\tilde{U}_{\delta}$. Finally, we note that since $M$ is compact, all the geodesics under consideration have length less than $\mathrm{Diam}(M,g_{0})$. As such, we can find a $\delta$ such that if $\tilde{U}_{\delta}$ is a symmetric family of curves about a geodesic, say $\tilde{\gamma}$, of length less than $\mathrm{Diam}(M)$, then
    \begin{equation}
        l(c)\leq l(\tilde{c})\leq l(\tilde{\gamma})+2\vare,
    \end{equation}
    where an error term $\vare$ is coming from two sources: the error roughly tangent to the geodesic $\tilde{\gamma}$ and the error in a direction perpendicular to $\tilde{\gamma}$.
    Since we chose $\tilde{\gamma}$ to satisfy $l(\tilde{\gamma})=l(\gamma)=d_{g_{0}}(x,y)$, we have the result.
\end{proof}


\begin{lem}\label{lem: distance potential estimate with error}
    Let $(M,g_{0})$ be a fixed Riemannian manifold with metric $g_{0}$, and let $h$ be another Riemannian metric on $M$ such that $h\geq g_{0}$. There exists a constant $C(M,g_0,\vare)>0$ such that if $U_{\vare}$ is any symmetric family of curves of width $\vare$ connecting two points, say $x$ and $y$, then we have
    \begin{equation}
        d_{h}(x,y)-d_{g_{0}}(x,y)-\vare\le C(M,g_{0},\vare)\left(\int_{M}\frac{|h-g_{0}|^{\frac{1}{2}}_{g_{0}}dV_{g_0}(z)}{d_{g_{0}}(x,z)^{m-1}}+\int_{M}\frac{|h-g_{0}|^{\frac{1}{2}}_{g_{0}}dV_{g_0}(z)}{d_{g_{0}}(y,z)^{m-1}}\right),
    \end{equation}
    where $C(M,g_0,\vare)\rightarrow\infty$ as $\vare\rightarrow0$.
\end{lem}
\begin{proof}
    Given $\vare>0$ we can use Lemma \ref{uniform distance estimate symmetric family of curves} to pick $\delta(\vare)=\delta>0$ such that if $U_{\delta}$ is a symmetric family of curves connecting any two points, say $x$ and $y$, then we have
    \begin{equation}
        \int_{[0,L]}|\Dot{\eta}|_{g_{0}}dt\leq d_{g_{0}}(x,y)+\vare,
    \end{equation}
    where $\eta$ is any of the curves foliating $U_{\delta}$. Combining the above with the fact that $d_{h}(x,y)\leq\int_{[0,l]}\sqrt{h(\Dot{\eta},\Dot{\eta})}$ gives the following inequality:
    \begin{align}
        d_{h}(x,y)-d_{g_{0}}(x,y)-\vare\leq\int_{[0,L]}\sqrt{h(\Dot{\eta},\Dot{\eta})}-\sqrt{g_{0}(\Dot{\eta},\Dot{\eta})}dt
        \\
        \le\int_{[0,L]}\sqrt{(h-g_{0})(\Dot{\eta},\Dot{\eta})}dt,
    \end{align}
    which holds since, by assumption, $|h|_{g_{0}}\geq |g_{0}|_{g_{0}}$. Now, we may estimate the last term in the above by
    \begin{equation}
        \int_{[0,L]}\sqrt{(h-g_{0})(\Dot{\eta},\Dot{\eta})}dt\leq\int_{[0,L]}|h-g_{0}|_{g_{0}}^{\frac{1}{2}}|\Dot{\eta}|_{g_{0}}dt.
    \end{equation}
    Thus, we have
    \begin{equation}
        d_{h}(x,y)-d_{g_{0}}(x,y)-\vare\leq \int_{[0,L]}|h-g_{0}|_{g_{0}}^{\frac{1}{2}}|\Dot{\eta}|_{g_{0}}dt.
    \end{equation}
    We may integrate both sides of the above by $\mu_{\delta}$ and apply Lemma \ref{Coarea Formula to Potential Integral} to get that
    \begin{equation}
        \omega_{m-1}\delta^{m-1}\left(d_{h}(x,y)-d_{g_{0}}(x,y)-\vare\right)
    \end{equation}
    is bounded by
    \begin{equation}
        C\left(M,g_{0}\right)\int_{M}|h-g_{0}|^{\frac{1}{2}}_{g_0}\left(d_{g_{0}}(x,y)^{1-m}+d_{g_{0}}(y,z)^{1-m}\right)dV_{g_0}(z).
    \end{equation}
    So, dividing both sides by $\omega_n\delta^{m-1}(\vare)$ gives
    \begin{align}
         d_{h}(x,y)-d_{g_{0}}(x,y)-\vare\leq& C(M,g_0)\delta(\vare)^{1-m}\int_{M}\frac{|h-g_{0}|^{\frac{1}{2}}_{g_{0}}dV_{g_0}(z)}{d_{g_{0}}(x,z)^{m-1}}
         \\
         &+C(M,g_0)\delta(\vare)^{1-m}\int_{M}\frac{|h-g_{0}|^{\frac{1}{2}}_{g_{0}}dV_{g_0}(z)}{d_{g_{0}}(y,z)^{m-1}}.     
    \end{align}
\end{proof}
We now turn the potential estimates above into integral estimates. However, we first need a lemma about potential functions on manifolds.
\begin{prop}\label{PotentialOperatorEstimate}
    Consider a closed Riemannian manifold $(M,g_0)$ and let $f$ be an element of $L^{p}(dVol_{g_{0}})$. Then, the function
    \begin{equation}
        \left(V_{g_{0}}f\right)(x):=\int_{M}f(z)d_{g_{0}}(x,z)^{1-m}dV_{g_{0}}(z),
    \end{equation}
    satisfies the bound
    \begin{equation}
        \|V_{g_{0}}f\|_{L^q}\leq C(M,g_{0},q)\|f\|_{L^{p}}
    \end{equation}
    for 
    \begin{equation}
        q<\frac{mp}{m-p}.
    \end{equation}
\end{prop}
\begin{proof}
    We give a similar argument to the potential estimate Lemma 7.12 in \cite{GT}. First we notice that for $R < \diam(M,g_0)$ we find
    \begin{align}
        (V_{g_0}1)&=\int_M d_{g_0}(x,z)^{1-m} dV_{g_0}(z)
        \\&=\int_{B_R(x)} d_{g_0}(x,z)^{1-m} dV_{g_0}(z)+\int_{M\setminus B_R(x)} d_{g_0}(x,z)^{1-m} dV_{g_0}(z)
        \\&\le C(m)\int_0^R r^{1-m} K^{\frac{1-m}{2}}\sinh^{m-1}(\sqrt{K}r) dr
        \\&\qquad +\int_{M\setminus B_R(x)} R^{1-m} dV_{g_0}(z)
        \\&\le C(M,g_0)R+\vol_{g_0}(M\setminus B_R(x)) R^{1-m} 
    \end{align}
    where $K<0$ is a lower bound on the sectional curvature of $(M,g_0)$. If $B_R(x)$ is no longer a topological ball then we can integrate in polar coordinates in the universal cover which will only increase the value of the integral over $B_R(x)$ with respect to $g_0$.
    Then by letting $R$ increase to the diameter we find $(V_{g_0}1) \le C(M,g_0) $.
    Then if we choose $t \ge 1$ so that 
    \begin{align}
        t^{-1}=1+q^{-1}-p^{-1},
    \end{align}
    we can obtain by a similar argument as above if $t< \frac{m}{m-1}$
    \begin{align}
        \int_M d_{g_0}(x,z)^{t(1-m)} dVol_{g_0}(z) \le C(M,g_0,t).
    \end{align} 
    Note that $t < \frac{m}{m-1}$ implies  $1+q^{-1}-p^{-1} > \frac{m-1}{m}$ and when we solve for $q$ we find $q < \frac{mp}{m-p}$.
    Now by writing 
    \begin{align}
        |f|d_{g_{0}}(x,z)^{1-m} = d_{g_{0}}(x,z)^{\frac{t}{q}(1-m)}|f|^{\frac{p}{q}} d_{g_{0}}(x,z)^{t(1-\frac{1}{p})(1-m)}|f|^{p(1-\frac{1}{t})},
    \end{align}
    and apply the generalized H\"{o}lder inequality with powers $q,\frac{1}{1-\frac{1}{p}}, \frac{1}{1-\frac{1}{t}}$ to find
    \begin{align}
        \int_{M}&|f(z)|d_{g_{0}}(x,z)^{1-m}dV_{g_{0}}(z)
        \\&\le \left(\int_{M}|f(z)|^pd_{g_{0}}(x,z)^{t(1-m)}dV_{g_{0}}(z) \right)^{\frac{1}{q}}
        \\&\cdot\left(\int_{M}d_{g_{0}}(x,z)^{t(1-m)}dV_{g_{0}}(z) \right)^{\frac{p-1}{p}}\left(\int_{M}|f(z)|^pdV_{g_{0}}(z) \right)^{\frac{t-1}{t}},
    \end{align}
    and hence we find
    \begin{align}
   & \|V_{g_0}f\|_{L^q}
      \\&= \left(\int_M \left(\int_M|f(z)|d_{g_{0}}(x,z)^{1-m}dV_{g_{0}}(z) \right)^qdV_{g_{0}}(x)\right)^{\frac{1}{q}}
        \\&\le \left(\int_{M}|f(z)|^pdV_{g_{0}}(z) \right)^{\frac{t-1}{t}}
        \\&\left(\int_M\int_{M}|f|^pd_{g_{0}}^{t(1-m)}dV_{g_{0}}(z)
        \left(\int_{M}d_{g_{0}}^{t(1-m)}dV_{g_{0}}(z) \right)^{\frac{q(p-1)}{p}}dV_{g_{0}}(x) \right)^{\frac{1}{q}}
        \\&\le \|f\|_{L^p}^{\frac{q-p}{q}} \sup_M \left(\int_{M}d_{g_{0}}(x,z)^{t(1-m)}dV_{g_{0}}(z) \right)^{\frac{p-1}{p}}
        \\&\cdot \left(\int_M|f(z)|^p\int_{M}d_{g_{0}}(x,z)^{t(1-m)}
        dV_{g_{0}}(x)dV_{g_{0}}(z) \right)^{\frac{1}{q}}
        \\&\le C(M,g_0,q)\|f\|_{L^p} .
    \end{align}
\end{proof}
\begin{lem}
    Let $(M,g_{0})$ be a fixed Riemannian manifold with metric $g_{0}$, and let $h$ be another Riemannian metric on $M$. Then, we have
    \begin{equation}
        \|d_{h}(x,y)\|_{L^{q}(M\times M, g_{0}\oplus g_{0})}\le C(M,g_0,q)\|h\|^{\frac{1}{2}}_{L^{\frac{p}{2}}(M,g_{0})},
    \end{equation}
    with $q<\frac{np}{n-p}$.
    
    Suppose $h\geq g_{0}$, and let $f_{\vare}(x,y)=\max\lbrace0,d_{h}(x,y)-d_{g_{0}}(x,y)-\vare\rbrace$. Then, we have
    \begin{equation}
        \|f_{\vare}(x,y)\|_{L^{q}(M\times M,g_{0}\oplus g_{0})}\leq C(\vare,M,g_0,q)\|h-g_{0}\|^{\frac{1}{2}}_{L^{\frac{p}{2}}(M,g_{0})},
    \end{equation}
    with $q<\frac{np}{n-p}$.
\end{lem}
\begin{proof}
    We prove only the first inequality, since the proof of the second is nearly identical. By applying Lemma \ref{lem: distance potential estimate} we can calculate that
    \begin{equation}
        \|d_{h}\|^q_{L^{q}}\leq C(\vare)\int_{M\times M}\left(\int_{M}\frac{|h|^{\frac{1}{2}}_{g_{0}}dV_{g_0}(z)}{d_{g_{0}}(x,z)^{m-1}}+\int_{M}\frac{|h|^{\frac{1}{2}}_{g_{0}}dV_{g_0}(z)}{d_{g_{0}}(y,z)^{m-1}}\right)^{q}dV_{g_0}(x)dV_{g_0}(y).
    \end{equation}
    The term on the right above may be estimated by
    \begin{equation}
        C(\vare,q)\int_{M\times M}\left(\int_{M}\frac{|h|^{\frac{1}{2}}_{g_{0}}dV_{g_0}(z)}{d_{g_{0}}(x,z)^{m-1}}\right)^{q}+\left(\int_{M}\frac{|h|^{\frac{1}{2}}_{g_{0}}dV_{g_0}(z)}{d_{g_{0}}(y,z)^{m-1}}\right)^qdV_{g_0}(x)dV_{g_0}(y).
    \end{equation}
    Next, we see that the above is bounded by
    \begin{align}
        &\vol_{g_{0}}(M)\int_{M}\left(\int_{M}\frac{|h|^{\frac{1}{2}}_{g_{0}}dV_{g_0}(z)}{d_{g_{0}}(x,z)^{m-1}}\right)^qdV_{g_0}(x)
        \\&\qquad +\vol_{g_{0}}(M)\int_{M}\left(\int_{M}\frac{|h|^{\frac{1}{2}}_{g_{0}}dV_{g_0}(z)}{d_{g_{0}}(y,z)^{m-1}}\right)^qdV_{g_0}(y)
        \\&=2\vol_{g_0}(M)\|(V_{g_0}|h|_{g_0}^{\frac{1}{2}})\|_{L^q}^q,
    \end{align}
    and so by applying Proposition \ref{PotentialOperatorEstimate} when $q < \frac{m-p}{mp}$ we see that
    \begin{equation}
        \|d_{h}(x,y)\|_{L^{q}(M\times M, g_{0}\oplus g_{0})}\le C(M,g_0,q)\|h\|^{\frac{1}{2}}_{L^{\frac{p}{2}}(M,g_{0})},
    \end{equation}
    where we are allowed to fix an $\varepsilon$ in order to remove the dependence on $\varepsilon$ in the constant. When repeating this argument for $f_{\varepsilon}$, the $\varepsilon$ dependence remains throughout the inequality.
\end{proof}
The following proposition is a minor adaptation of Frostman's lemma to a Riemannian manifold.
\begin{prop}\label{Manifold Frostman's Lemma}
    Let $(M,g_{0})$ be a compact closed Riemannian manifold, let $s>0$, and suppose $A$ is a Borel set such that $H^{s}_{g_{0}}(A)>0$. Then, there exists a probability measure, say $\mu$, supported in a subset of $A$, and a positive constant $c$ such that for all $x$ in the support of $\mu$ we have
    \begin{equation}
        \mu\left(B_{r}(x)\right)\leq cr^{s}, \quad \text{ for all } 0 < r < \diam(M,g_0).
    \end{equation}
\end{prop}
\begin{proof}
    Since $M$ is compact, we can find a finite collection of charts on $M$, say
    \begin{equation}
        \lbrace\phi_{i}:B_{r_i}(0)\subset\R^{m}\rightarrow B_{r_i}(x_{i})\subset M\rbrace^{N}_{i=0},
    \end{equation}
    such that there exists a positive constant $C$ with the property that
    \begin{equation}\label{coordinate metric comparison}
        \frac{1}{C^2}\leq\frac{\phi_{i}^{*}g_{0}}{g_{E}}\leq C^2,
    \end{equation}
    and $B_i=\phi_{i}\left(B_{\frac{r_i}{4}}(0)\right)$ covers $M$, where $g_{E}$ is the flat background metric on the coordinates.
    
    Since $A=\bigcup^{N}_{i=0}B_{i}\bigcap A$, we have
    \begin{equation}
        H^{s}_{g_{0}}(A)\leq\sum^{N}_{i=0}H^{s}_{g_0}\left(B_i\bigcap A\right).
    \end{equation}
    Furthermore, since each $B_{i}$ is a closed set, each $B_{i}\bigcap A$ is Borel. Without loss of generality, we may assume that $H^{s}_{g_0}(B_{0}\bigcap A)>0$. This gives us
    \begin{equation}
        0<H^{s}_{g_0}\left(B_{0}\bigcap A\right)=H^{s}_{\phi^{*}_{0}g_{0}}\left(\phi^{-1}_{0} \left(B_{0}\bigcap A\right)\right)\leq C^{s} H^{s}_{g_{E}}\left(\phi^{-1}_{0} \left(B_{0}\bigcap A\right)\right).
    \end{equation}
    Thus, we may apply the usual Frostman's lemma, see Theorem 8.8 in \cite{M}, to the set $\phi^{-1}_{0}\left(B_{0}\bigcap A\right)$ to obtain a probability measure supported on a subset of
    \begin{equation}
        \phi^{-1}_{0}\left(B_{0}\bigcap A\right)\subset B_{\frac{r_0}{4}}(0),
    \end{equation}
    say $\tilde{\mu}$, and a constant $c>0$ such that for any $x$ in the support of $\tilde{\mu}$ we have
    \begin{equation}
        \tilde{\mu}\left(B^{g_{E}}_{r}(x)\right)\leq cr^{s}.
    \end{equation}
    Furthermore, by definition, for every $x$ in the support of $\tilde{\mu}$, we have that $B^{g_{E}}_{\frac{r_0}{2}}(x)$ contains the support of $\tilde{\mu}$:
    \begin{equation}
        \tilde{\mu}\left(B^{g_{E}}_{\frac{r_0}{2}}(x)\right)=1.
    \end{equation}
    Now, from Equation (\ref{coordinate metric comparison}), we have
    \begin{equation}
        B^{g_{E}}_{\frac{r}{C}}\subset B^{\phi^{*}_{0}g_0}_{r}(x)\subset B^{g_{E}}_{Cr}.
    \end{equation}
    Thus, we can conclude that
    \begin{equation}
        \tilde{\mu}\left( B^{\phi^{*}_{0}g_0}_{r}(x)\right)\leq\tilde{\mu}\left(B^{g_{E}}_{Cr}\right)\leq c C^{s}r^{s},
    \end{equation}
    for all $x$ in the support of $\tilde{\mu}$. Finally, we let $\mu=\phi^{\#}_{0}\tilde{\mu}$.
\end{proof}
Now, we need to understand how the probability measures constructed above behave with respect to the potentials we are working with.
\begin{lem}\label{L1 bounded map}
    Let $\mu$ be a probability measure such that for all $x$ in the support of $\mu$ we have
    \begin{equation}
        \mu(B_{r}(x))\le \kappa r^s,
    \end{equation}
    where $\kappa$ is nonnegative, and $s$ is strictly positive. Let $0\le t<s$, then the map
    \begin{equation}
        f:\rightarrow\int_M\frac{f(y)}{d_{g_{0}}(x,y)^t}dV_{g_{0}}(y)
    \end{equation}
    is a bounded map from $L^{1}(dVol_{g_{0}})$ to $L^{1}(\mu)$:
    \begin{equation}
        \int_M\left|\int_M\frac{f(y)}{d_{g_{0}}(x,y)^t}dV_{g_{0}}(y)\right|d\mu(x)\le C(s,t,\kappa)\|f\|_{L^{1}(dVol_{g_{0}})},
    \end{equation}
    Where $C(s,t)$ tends to infinity as $t$ tends towards $s$.
\end{lem}
\begin{proof}
    We may assume without loss of generality that $f$ is nonnegative. Thus, using the Fubini-Tonelli theorem, we have
    \begin{equation}
        \int_M\int_M\frac{f(y)}{d_{g_{0}}(x,y)^t}dV_{g_{0}}(y)d\mu(x)=\int_M f(y)\int_M\frac{1}{d_{g_{0}}(x,y)^t}d\mu(x)dV_{g_{0}}(y).
    \end{equation}
    Now, using the layer cake formula for integrating, we have
    \begin{equation}
        \int_{M}\frac{1}{d_{g_{0}}(x,y)^{t}}d\mu(x)=\int^{\infty}_{0}\mu\left(\lbrace x: d_{g_{0}}(x,z)^{-t}\ge u\rbrace\right)du.
    \end{equation}
    Now, observe that $\lbrace x:d_{g_{0}}(x,z)^{-t}\ge u\rbrace=\lbrace x:d_{g_{0}}(x,z)\le u^{\frac{-1}{t}}\rbrace=B(z,u^{\frac{-1}{t}})$. So, the above becomes
    \begin{equation}
        \int^{\infty}_{0}\mu\left(B(z,u^{\frac{-1}{t}})\right)du.
    \end{equation}
    Since $M$ has finite diameter, we observe that for all $u\le \text{diam}(M)^{-t}$, we have $\mu(B(z,u^{\frac{-1}{t}}))=1$. Thus, the above is bounded by
    \begin{equation}
        \text{diam}(M)^{-t}+\kappa \int^{\infty}_{\text{diam}(M)^{-t}}u^{\frac{-s}{t}}du,
    \end{equation}
    which is bounded since $s>t$. Thus, we see that
    \begin{equation}
        \int_M\int_M\frac{f(y)}{d_{g_{0}}(x,y)^{t}}dV_{g_{0}}(y)d\mu(x)\le C(s,t,\kappa)\|f_{j}\|_{L^{1}(M,g_{0})}.
    \end{equation}
\end{proof}

\begin{lem}\label{finite L1 sum Hausdorff estimate}
    Let $f_j$ be non-negative Borel measurable functions such that 
    \begin{equation}
        \sum^{\infty}_{j=1}\|f_{j}\|_{L^{1}(dV_{g_{0}})}<\infty,   
    \end{equation}
    and let $t > 0$ and $B(\delta)$ be defined by 
    \begin{equation}
        B(\delta)=\left\lbrace x:\limsup_{j\rightarrow\infty}\int_M\frac{f_{j}(y)}{d_{g_{0}}(x,y)^{t}}dV_{g_{0}}(y)\ge\delta\right\rbrace.
    \end{equation}
    Then, the Hausdorff dimension of $B(\delta)$ with respect to $g_{0}$ is less than or equal to $t$.
\end{lem}
\begin{proof}
    Since we will need to apply Frostman's Lemma, it is important to show that the set $B(\delta)$ is Borel. To see that $B(\delta)$ is Borel, first observe that the function $\frac{f_j}{d_{g_{0}}(x,y)^t}$ is Borel, since both the numerator and the denominator are Borel measurable. Thus, by the Fubini-Tonelli theorem, it follows that
    \begin{equation}
        x\rightarrow\int_M\frac{f_j}{d_{g_{0}}(x,y)^t}dV_{g_{0}}(y)
    \end{equation}
    is a Borel measurable function. It is standard that taking the limit supremum of measurable functions is itself measurable, see for example Proposition 2.7 in \cite{F}. Thus, it follows that $B(\delta)$ is in fact a Borel set.
    
    Suppose $\mathrm{dim}_{H_{g_{0}}}\left(B(\delta)\right)=s_0>t$, then it follows by definition that $\mathcal{H}^{t}_{g_{0}}\left(B(\delta)\right)=\infty$. Since $B(\delta)$ is Borel, we may appeal to the version of Frostman's lemma in Proposition \ref{Manifold Frostman's Lemma} in order to find a Radon probability measure, say $\mu$, whose support is contained in $B(\delta)$, and which satisfies
    \begin{equation}
        \mu(B_{r}(x))\le cr^{t},
    \end{equation}
    for all $x$ in $B(\delta)$. By definition of $B(\delta)$, we have for every $x$ in $B(\delta)$ that
    \begin{equation}
        \delta\le\limsup_{j\rightarrow\infty}\int_M\frac{f_{j}}{d_{g_{0}}(x,y)^{t}}dV_{g_{0}}\le\int_M\frac{1}{d_{g_{0}}(x,y)^{t}}\sum^{\infty}_{k=l}f_{k}(y)dV_{g_{0}}(y),
    \end{equation}
    for all $l$. Integrating the above by $\mu$, and using Lemma \ref{L1 bounded map}, we get
    \begin{equation}
        \delta\le C(s,t)\sum_{k=l}^{\infty}\|f_{k}\|_{L^{1}(dV_{g_{0}})},
    \end{equation}
    for all $l$. This is a contradiction, since the right hand side can be made arbitrarily small by sending $l$ to infinity. Thus, we see that $s_0\leq t$.
\end{proof}
\begin{lem}\label{applied reverse hoelder ineq}
    Let $\delta>0$, and let $x$ be in $M$ and suppose that $f$ is a non-negative function such that
    \begin{equation}
        \int_{M}\frac{f}{d_{g_{0}}(x,y)^{m-1}}dV_{g_0}\ge\delta.
    \end{equation}
    Then, for $p > 1$, $t>m-p$ we have that
    \begin{equation}
        \int_M\frac{f^p}{d_{g_{0}}(x,y)^t}dV_{g_0}\ge C(g_{0},t,p)\delta^p,
    \end{equation}
    where $C(g_{0},t,p)$ only depends on $g_{0}$, $t$, and $p$.
\end{lem}
\begin{proof}
    The proof follows from a simple application of a reverse H{\"o}lder inequality:
    \begin{align}
        \int_{M}\frac{f^p(z)}{d_{g_{0}}(x,z)^t}dV_{g_0}(z)&=\int_{M}\frac{f^p(z)}{d_{g_{0}}(x,z)^{p(m-1)}}d_{g_{0}}(x,z)^{p(m-1)-t}dV_{g_0}(z)
        \\
        &\ge\left(\int_{M}\frac{f(z)}{d_{g_{0}}(x,z)^{m-1}}dV_{g_0}(z)\right)^{p}
        \\&\quad \cdot\left(\int_{M} d_{g_{0}}(x,z)^{-\frac{1}{p-1}(p(m-1)-t)}dV_{g_0}(z)\right)^{1-p}.
    \end{align}
    We now observe that for $t>m-p$, we have that
    \begin{equation}
        \frac{p(m-1)-t}{p-1}<m.
    \end{equation}
    Thus, since $p-1>0$ it follows that
    \begin{equation}
        \int_M d_{g_{0}}(x,z)^{\frac{-1}{p-1}(p(m-1)-t)}dV_{g_0}(z)\le C(g_{0},t,p)<\infty.
    \end{equation}
\end{proof}

\begin{prop}\label{Lp Hausdorff dimension control}
    Let $\delta>0$, and let $f_{j}$ be a sequence of non-negative functions tending to zero in $L^{p}(g_{0})$.
    If $f_{j}$ is such that
    \begin{equation}
        \sum^{\infty}_{j=1}\|f_{j}\|^p_{L^{p}(dV_{g_{0}})}<\infty,
    \end{equation}
    and $B(\delta)$ is defined by 
    \begin{equation}
        B(\delta)=\left\lbrace x:\limsup_{j\rightarrow\infty}\int_M\frac{f_{j}(y)}{d_{g_{0}}(x,y)^{m-1}}dV_{g_{0}}(y)\ge\delta\right\rbrace.
    \end{equation}
    then we have
    \begin{equation}
        \mathrm{dim}_{H_{g_{0}}}\left(B(\delta)\right)\le m-p.
    \end{equation}
\end{prop}
\begin{proof}
    First, apply Lemma \ref{applied reverse hoelder ineq} to the sequence $f_{j}$ for any $t>m-p$. Then apply Lemma \ref{finite L1 sum Hausdorff estimate} to see that
    \begin{equation}
        \mathrm{dim}_{H_{g_{0}}}\left(B(\delta)\right)\le t.
    \end{equation}
    Letting $t$ tend towards $m-p$ gives the result.
\end{proof}

\section{proofs of the main results}
In this section we give the proofs of the main theorems of this paper stated in the introduction.

\begin{thm}
Let $M$ be a connected, closed, and oriented manifold, $g_0$ a smooth Riemannian manifold and $g_j$ a sequence of continuous Riemannian manifolds. If
\begin{align}
    g_j(v,v) \ge \left(1-\frac{1}{j} \right) g_0(v,v), \quad \forall p\in M, v \in T_pM,
\end{align}
and
\begin{align}
    \|g_j-g_0\|_{L^p_{g_0}(M)}\rightarrow 0,\quad p\ge\frac{1}{2},
\end{align}
then there exists a subsequence $g_k$ and there exists $U \subset M$ measurable so that $\vol_{g_0}(U)=\vol_{g_0}(M)$  and $\mathrm{dim}_{H_{g_0}}(M \setminus U)\le m-2p$ and
\begin{align}
    d_k(p,q) \rightarrow d_0(p,q), \quad \forall p,q \in U.
\end{align}
Furthermore, for every $\varepsilon > 0$ we can find a measurable set $U_{\varepsilon}\subset M$ so that $\vol_{g_0}(M \setminus U_{\varepsilon})< \varepsilon$ and 
\begin{align}
    d_k(p,q) \rightarrow d_0(p,q),
\end{align}
uniformly $\forall p,q \in U_{\varepsilon}$.
\end{thm}
\begin{proof}
    Since $\|g_{j}-g_{0}\|_{L^{p}_{g_{0}}(M)}\rightarrow 0$, we can find a sub-sequence $g_k$ so that 
    \begin{equation}
        \sum^{\infty}_{j=1}\|g_{k}-g_{0}\|^p_{L^{p}_{g_{0}}(M)}<\infty.
    \end{equation}
    Now, let $f_{k}=\sqrt{|g_{k}-g_{0}|_{g_{0}}}$. Thus the sequence $f_{k}$ satisfies the hypothesis of Proposition \ref{Lp Hausdorff dimension control}. Therefore, for any $\delta>0$, we have that $\mathrm{dim}_{H_{g_0}}B(\delta)\le m-2p$.
    
    Let $D=\bigcup^{\infty}_{n=1} B\left(\frac{1}{n}\right)$. Since the Hausdorff dimension of each $B(\frac{1}{n})$ is bounded above by $m-2p$, it follows that the Hausdorff dimension of $D$ is less than $m-2p$. Let $x$ and $y$ be in $U:=M\setminus D$. Then, by the definition of $D$, we see that for $x,y$ we have
    \begin{equation}
        \lim_{k\rightarrow\infty}\int_M\frac{f_{k}(z)}{d_{g_0}(x,z)^{m-1}}+\frac{f_{k}(z)}{d_{g_0}(y,z)^{m-1}}dV_{g_0}(z)=0.
    \end{equation}
    From Lemma \ref{lem: distance potential estimate with error}, we have for any $\vare>0$ that
    \begin{equation}
        d_{g_{k}}(x,y)-d_{g_{0}}(x,y)-\vare\le C(\vare)\left(\int_M\frac{f_{k}(z)}{d_{g_0}(x,z)^{m-1}}+\frac{f_{k}(z)}{d_{g_0}(y,z)^{m-1}}dV_{g_0}(z)\right).
    \end{equation}
    Since $x$ and $y$ are in $U=M\setminus D$, it follows that the right hand side of the above tends towards zero as $k\rightarrow\infty$. Thus, for each $x$ and $y$ in $U$, and each $\vare>0$, we have
    \begin{equation}
        \limsup_{k\rightarrow\infty}d_{g_{k}}(x,y)\le d_{g_{0}}(x,y)+\vare.
    \end{equation}
    Finally, we let $\vare$ tend to zero. Thus, for $x$ and $y$ in $U$, we have
    \begin{equation}
        \limsup_{k\rightarrow\infty}d_{g_{k}}(x,y)\le d_{g_0}(x,y).
    \end{equation}
    However, the distance below estimate for each $g_k$ shows that we in fact have $\displaystyle\lim_{k \rightarrow \infty} d_{g_k}(x,y)=d_{g_0}(x,y)$.
    Then by Egeroff's theorem, for any $\varepsilon>0$ we can find a set $U_{\varepsilon}$ so that $ U_{\varepsilon} \subset M$, $\vol(M \setminus U_{\varepsilon})<\varepsilon$ and $d_k$ converges uniformly on $U_{\varepsilon}$, as desired.
\end{proof}

We now give a proof for Theorem \ref{TorusThm} using results from the first named author in \cite{Allen-Tori}. Let us recall the statement of Theorem \ref{TorusThm}:
\begin{*thm}
Consider a sequence of Riemannian $n-$manifolds $M_j=(\mathbb{T}^n,g_j)$, $n \ge 3$ satisfying 
 \begin{align}
R_{g_j} \ge -\frac{1}{j}, \,\,\,  \,\,\, \diam(M_j) \le D_0, \,\,\,  \,\,\, \vol(M_j) \le V_0.
\end{align}
Let $g_0$ be a flat torus where $\mathbb{T}^m_0=(\mathbb{T}^m,g_0)$ and assume that $M_j$ is conformal to $\tilde{M}_{0,j}=(\mathbb{T}^m,\tilde{g}_{0,j})$, a metric with constant zero or negative scalar curvature and unit volume, i.e.  $g_j = e^{2f_j} g_{0,j}$. Furthermore, assume that
\begin{align}
\tilde{g}_{0,j} \rightarrow g_0 \text{ in } C^1,
\end{align}
and
\begin{align}
\int_{\mathbb{T}^m} e^{-2f_j} d V_{\tilde{g}_{0,j}} \le C,
\end{align} 
then for any $\varepsilon >0$ there exists a subsequence and a set $U_{\varepsilon} \subset \Tor^m$ so that  $\vol(U_{\varepsilon})=\vol(\Tor^m)-\varepsilon$ and
\begin{align}
  d_k(q_1,q_2) \rightarrow d_{\infty}(q_1,q_2),
\end{align}
uniformly on $U_{\varepsilon}$ where $d_{\infty}$ is the distance function for $\bar{\mathbb{T}}_0^m = (\Tor^m,\bar{g}_0=c_{\infty}^2g_0)$, $c_{\infty}^2=\displaystyle\lim_{k \rightarrow \infty}(\overline{e^{-f_k}})^{-2}= \lim_{k \rightarrow \infty} \left( \dashint_{\mathbb{T}^m} e^{-f_k}dV_{\tilde{g}_{0,j}} \right)^{-2}$.
\end{*thm}
\begin{proof}[Proof of Theorem \ref{TorusThm}]
By rewriting the scalar curvature of $M_j$ in terms of $M_{0,j}$, one can obtain the following PDE inequality:
\begin{align}\label{Main PDE Inequality}
-(n-1)\Delta^{g_{0,j}} e^{-2f_j} + (n-1)(n+2) |\nabla^{g_{0,j}} e^{-f_j}|^2  \le \frac{1}{j}.
\end{align}
In the proof of Theorem 1.4 of \cite{Allen-Tori} it is shown how to use the $C^1$ convergence of $g_{0,j} \rightarrow g_0$ to rewrite this PDE inequality in terms of $g_0$, and after integrating one obtains the estimate
\begin{align}\label{SobolevBound}
\int_{\mathbb{T}^m} |\nabla^{g_0} e^{-f_j}|^2  dV_{g_0} \le \frac{C}{j}.
\end{align}
Hence by the Poincare inequality we find
\begin{align}\label{Poincare Consequence}
C_P \int_{\Tor^n}|e^{-f_j} - \overline{e^{-f_j}}|^2dV_{g_{0}} \le \int_{\Tor^n} |\nabla^{g_{0}} e^{-f_j}|^2  dV_{g_{0}} &\le \frac{C}{j},
\end{align}
which implies $L^2$ convergence of $e^{-f_j}$ to its average as well as pointwise a.e. convergence on a subsequence of $e^{-f_k}$ to its average,
\begin{align}
\overline{e^{-f_k}} = \dashint_{\Tor^n}e^{-f_j}dV_{g_{0,k}},
\end{align}
which is non-zero and well defined by Lemma 4.2 of \cite{Allen-Tori} and where we have taken a further subsequence if necessary so that $\displaystyle \lim_{k \rightarrow \infty} \overline{e^{-f_k}}$ is well defined. This also implies convergence in $L^2$ norm for any measurable set $U \subset \mathbb{T}^n$
\begin{align}\label{ImportantConvergenceEq}
\lim_{k\rightarrow \infty}\int_{U} e^{-2f_k} dV_{g_{0,k}}=\lim_{k\rightarrow \infty} \left(\overline{e^{-f_k}}\right)^2\vol_{g_{0,k}}(U).
\end{align} 
In fact, by the reverse triangle inequality for norms, if we have two functions $h,k:\mathbb{T}^n\rightarrow \R$ we find
\begin{align}
\left|\|h\|_{L^2_g(U)}-\|k\|_{L^2_g(U)}\right| \le \|h-k\|_{L^2_g(U)} \le \|h-k\|_{L^2_g(\mathbb{T}^n)},
\end{align}
for any Riemannian metric $g$ and hence \eqref{ImportantConvergenceEq} is uniform in $U$ and in particular uniform in $B_{M_{0,j}}(x,r)$ for all $x \in \mathbb{T}^n$, $0 < r \le R$.
In the proof of Theorem 1.4 of \cite{Allen-Tori} it is then shown how to use the $C^1$ convergence of $g_{0,j}\rightarrow g_0$ to rewrite \eqref{Main PDE Inequality} to obtain the PDE inequality
\begin{align}\label{Essential PDE Rewrite 2}
- \Delta^{g_0} e^{-2f_j}  &\le \frac{C}{j}.
\end{align}
Then, as in Lemma 4.1 of \cite{Allen-Tori}, we can subtract an auxiliary function from $e^{-2f_j}$ to find a subharmonic function to apply  the mean value inequality to estimate the conformal factor from below by 
\begin{align}
\lp\limsup_{j \rightarrow \infty}\max_{x \in \mathbb{T}^n}\dashint_{B_{M_{0,j}}(x,r)} e^{-2f_j} dV_{g_{0,j}}\rp^{-1} \le \liminf_{j \rightarrow \infty} \min_{\mathbb{T}^n} e^{2f_j}, \quad r \le R,
\end{align} 
which we have just argued is well defined and converging to a constant. Hence by combining \eqref{ImportantConvergenceEq} with the obtained lower bound on the sequence of conformal factors we are able to find
\begin{align}
    c_{\infty}^2\left( 1 - \frac{C}{j} \right) g_0(v,v)&= \left( 1 - \frac{C}{j} \right) \bar{g}_0(v,v)
    \\&\le g_k(v,v), \quad \forall p \in M, v \in T_pM.
\end{align}
Hence by Theorem \ref{MainTheorem} we find pointwise convergence of $d_{g_k}$ to $d_{\infty}$ away from a set $S$ so that $\vol(S)=0$ and for any $\varepsilon>0$ we can find a set $S_{\varepsilon}$ so that $S \subset S_{\varepsilon}$, $\vol(S_{\varepsilon})<\varepsilon$ and $d_k$ converges uniformly on $U_{\varepsilon}=\Tor^m \setminus S_{\varepsilon}$, as desired.
\end{proof}

\bibliographystyle{alpha}
\bibliography{bibliography}
\end{document}